\documentclass[12pts]{article}
\usepackage{amsmath,amsfonts,amssymb,amsthm,mathtools}
\usepackage{microtype}
\usepackage{lipsum}
\usepackage{url,hyperref}
\usepackage{indentfirst}
\hypersetup{colorlinks, linkcolor={red}, citecolor={blue}, urlcolor={black}}
\usepackage{comment}
\usepackage[margin=1.25in]{geometry}
\usepackage{enumerate,enumitem}
\usepackage{subcaption}
\usepackage{tikz}
\usepackage{array,multirow}

\allowdisplaybreaks

\newcommand\floor[1]{\left\lfloor#1\right\rfloor}

\newtheorem{theorem}{Theorem}[section]
\newtheorem{lemma}[theorem]{Lemma}

\newtheorem{corollary}[theorem]{Corollary}

\newtheorem{question}[theorem]{Question}

\newcounter{propcounter}

\title{Lattice paths enumerations weighted by ascent lengths}
\author{Jun Yan\thanks{Mathematics Institute, University of Oxford, UK. Email: {\tt jun.yan@maths.ox.ac.uk}. JY was supported by the Warwick Mathematics Institute CDT and funding from the UK EPSRC (Grant number: EP/W523793/1) when this work was done.}}
\date{}

\begin{document}

\maketitle
\begin{abstract}
Recent work of the author connected several parking function enumeration problems to enumerations of Catalan paths with respect to certain weight functions that are expressed in terms of the ascent lengths. Motivated by this, we generalise and solve analogous weighted enumeration problems for a large family of lattice paths and weight functions, and discuss their connections with other enumeration problems and OEIS entries. 
\end{abstract}

\section{Introduction}\label{sec:intro}
Given a set $\mathcal{A}$ of allowed steps, let $\mathcal{P}_n$ be the set of lattice paths from $(0,0)$ to $(n,0)$, never going below the $x$-axis, and each of whose steps are in $\mathcal{A}$. For example, if $\mathcal{A}=\{(1,1),(1,-1)\}$ then $\mathcal{P}_{2n}$ is the well-known set of Catalan paths of size $n$, if $\mathcal{A}=\{(1,1),(2,0),(1,-1)\}$ then $\mathcal{P}_{2n}$ is the set of Schr\"{o}der paths of size $n$, and if $\mathcal{A}=\{(1,1),(1,0),(1,-1)\}$ then $\mathcal{P}_{n}$ is the set of Motzkin paths of size $n$. The enumeration of $\mathcal{P}_n$ is very well-studied, with many known algebraic and bijective proofs for the special cases mentioned above, as well as the kernel method (see for example \cite{BF,FS}) that handles the general cases. 

In a recent paper \cite{Y}, the author discovered connections between several weighted enumerations of the set $\mathcal{C}_n$ of Catalan paths of size $n$ with certain parking function enumeration problems. These weights are defined in terms of the length of each maximal block of consecutive ascents in the Catalan path, which we refer to as \textit{ascent lengths} from now on. By enumerating $\mathcal{C}_n$ with respect to these weights using generating function techniques, the author was able to provide explicit formulas solving these parking function enumeration problems. 

In this paper, motivated by these connections and results in \cite{Y}, we consider more generally a large family of lattice paths and weight functions described as follows.

Let $I,J,K$ be three disjoint countable index sets and let $\mathcal{U}=\{(r_i,s_i)\mid i\in I\}, \mathcal{V}=\{(r_i,s_i)\mid i\in J\}, \mathcal{W}=\{(r_i,-1)\mid i\in K\}$ be three sets of allowed steps, where $r_i\in\mathbb{Z}_+$ for all $i\in I\cup J\cup K$ and $s_i\in\mathbb{Z}_{\geq0}$ for all $i\in I\cup J$. For all $n\in\mathbb{Z}_{\geq0}$, let $\mathcal{P}_n$ be the set of lattice paths from $(0,0)$ to $(n,0)$, never going below the $x$-axis, and each of whose steps is in $\mathcal{U}\cup\mathcal{V}\cup\mathcal{W}$. Note that down-steps for paths in $\mathcal{P}_n$ can only have vertical size 1.

For each $P\in\mathcal{P}_n$, let the vector $\mathbf{u}(P)$ record in order the length of each maximum block of consecutive steps in $P$ that belong to $\mathcal{U}$. We are interested in the sum
\[p_n=\sum_{P\in\mathcal{P}_n}\prod_{i=1}^{|\mathbf{u}(P)|}f(\mathbf{u}(P)_i),\]
for some weight function $f:\mathbb{Z}_+\to\mathbb{C}$. Note that if $f(\ell)=1$ for all $\ell\in\mathbb{Z}_+$, then $p_n$ simply enumerates the set $\mathcal{P}_n$. 

As an example, for $\mathcal{U}=\{(1,1)\}, \mathcal{V}=\emptyset, \mathcal{W}=\{(1,-1)\}$ and the function $f(\ell)=\ell$, $\mathcal{P}_{2n}$ is the set of Catalan paths of size $n$, and $p_{2n}$ is the weighted enumeration of $\mathcal{P}_{2n}$, where each $P\in\mathcal{P}_{2n}$ is weighted by the product of the ascent lengths in $P$. In particular, for $P=UUWUUUWWUWWW\in\mathcal{P}_{12}$, where $U$ denotes a $(1,1)$-step and $W$ denotes a $(1,-1)$-step, the three ascent lengths are recorded in the vector $\mathbf{u}(P)=(2,3,1)$, and $P$ has weight $2\times3\times1=6$. In \cite{Y}, it was shown that $p_{2n}$ coincides with the number of parking functions of size $n$ whose parking outcomes, viewed as a permutation of size $n$, avoid the pattern 123.  

In Section \ref{sec:main}, we first use a certain lattice path decomposition to prove a ``master theorem'' that provides a functional equation that is satisfied by the generating function $P(x)$ of the sequence $p_n$. Then, for two large families of lattice paths and weight functions, we obtain formulas for $p_n$ by using Lagrange's Implicit Function Theorem.

In Section \ref{sec:catalan}, we specialise the results in Section \ref{sec:main} to generalised Catalan paths and certain weight functions to obtain explicit formulas for $p_n$ in these cases, and discuss their connections to various enumeration problems, as well as providing first combinatorial interpretations to several OEIS sequences. Here, we recover several results in \cite{Y}. In Section \ref{sec:schroeder}, we give analogous treatments to generalised Schr\"{o}der paths, as defined recently by Yang and Jiang in \cite{YJ}. Finally, in Section \ref{sec:motzkin} we propose a new and natural generalisation of Motzkin paths analogous to the generalisations of Catalan paths and Schr\"{o}der paths used in the previous sections. For each $k\in\mathbb{Z_+}$, our definition of $(2k-1)$-Motzkin paths coincides with a recent definition of $k$-Motzkin paths given by Yang, Zhang and Yang in \cite{YZY}, and the framework of this paper applies to give analogous weighted enumeration results on them as well. However, our method fails for the $2k$-Motzkin paths under our definition, due to the presence of down-steps of vertical size larger than 1. Nevertheless, we are able to use the kernel method to at least obtain an explicit formula counting the (unweighted) number of $2k$-Motzkin paths.

We remark that the decomposition we used to prove our ``master theorem'' was also used by Mansour and Sun in a series of papers \cite{MS2,MS1,Sun}. There, for a (different) generalisation of the Catalan paths and Motzkin paths, they studied the related statistics $\alpha_j(P)$ counting the number of ascents of length $j$ in a lattice path $P$, and expressed the corresponding generating functions in terms of the Bell polynomials and the potential polynomials. We refer interested readers to their paper for more details.

\section{General results}\label{sec:main}
\begin{theorem}\label{thm:master}
Let $U(x,y)=\sum_{i\in I}x^{r_i}y^{s_i}$, $V(x,y)=\sum_{i\in J}x^{r_i}y^{s_i}$, $W(x)=\sum_{i\in K}x^{r_i}$, $F(x)=1+\sum_{n=1}^\infty f(\ell)x^\ell$ and $P(x)=\sum_{n=0}^\infty p_nx^n$. Then, 
\[P(x)=F(U(x,W(x)P(x)))(1+P(x)V(x,W(x)P(x))).\]
\end{theorem}
\begin{proof}
Let $\mathcal{P}=\cup_{n\geq0}\mathcal{P}_n$, so $P(x)$ is the weighted generating function for $\mathcal{P}$. For each $P\in\mathcal{P}$, let $\ell\geq0$ be maximal such that the first $\ell$ steps of $P$ are all in $\mathcal{U}$, and call these $\ell$ steps $U_1,\ldots,U_\ell$. For each $i\in[\ell]$, let $u_i$ be the vertical step size of $U_i$, so in particular $P$ reaches height $u:=\sum_{i=1}^\ell u_i$ after these $\ell$ steps. By maximality, the next step in $P$ (if any) after these $\ell$ steps is not in $\mathcal{U}$.

\textbf{Case 1.} The next step is either in $\mathcal{W}$ or doesn't exist. If $u=0$, there is no more step as any step in $\mathcal{W}$ will go below the $x$-axis. Now assume $u\geq 1$. For every $i\in[u]$, let $W_i$ be the first step in $P$ that goes from height $i$ to height $i-1$. In particular, $W_u$ immediately follows $U_\ell$ in $P$. Note that from definition, for every $i\in[u-1]$, the portion of $P$ between $W_{i+1}$ and $W_i$ is itself a lattice path in $\mathcal{P}$, and so is the portion of $P$ after $W_1$. 

Therefore, $P$ decomposes as $U_1\cdots U_\ell W_uP_uW_{u-1}\cdots W_1P_1$ for some $P_1,\ldots,P_u\in\mathcal{P}$, where we interpret this as simply $U_1\cdots U_\ell$ when $u=0$. Moreover, this process is evidently reversible, so the weighted generating function for all $P\in\mathcal{P}$ of this type is
\[1+\sum_{\ell=1}^\infty f(\ell)\left(\sum_{i\in I}x^{r_i}W(x)^{s_i}P(x)^{s_i}\right)^\ell=F(U(x,W(x)P(x))).\]

\textbf{Case 2.} The next step is $V_1$ in $\mathcal{V}$ of vertical size $v_1$. For every $i\in[u+v_1]$, let $W_i$ be the first step in $P$ that goes from height $i$ to height $i-1$. Like above, from definition, $P$ decomposes as $U_1\cdots U_\ell V_1P_{u+v_1+1}W_{u+v_1}P_{u+v_1}\cdots W_1P_1$ for some $P_1,\ldots,P_{u+v_1+1}\in\mathcal{P}$. Again, the process is reversible, so the weighted generating function for all $P\in\mathcal{P}$ of this type is
\begin{align*}
&\phantom{{}={}}\sum_{i\in J}x^{r_i}W(x)^{s_i}P(x)^{s_i+1}+\sum_{\ell=1}^\infty f(\ell)\left(\sum_{i\in I}x^{r_i}W(x)^{s_i}P(x)^{s_i}\right)^\ell\left(\sum_{i\in J}x^{r_i}W(x)^{s_i}P(x)^{s_i+1}\right)\\
&{}=F(U(x,W(x)P(x)))P(x)V(x,W(x)P(x)).
\end{align*}

Putting these two cases together, we get
\begin{align*}
P(x)&=F(U(x,W(x)P(x)))+F(U(x,W(x)P(x)))P(x)V(x,W(x)P(x))\\
&=F(U(x,W(x)P(x)))(1+P(x)V(x,W(x)P(x))),
\end{align*}
as required.
\end{proof}

To obtain more explicit results, we will make use of the following well-known Lagrange's Implicit Function Theorem, whose proof can be found in \cite{GJ}. This allows us to extract the coefficients of a power series satisfying certain functional equations without solving for the power series itself. In what follows, we use $R[[x]]$ to denote the ring of formal power series with coefficients in $R$, and use $[x^n]P(x)$ to denote the coefficient of the $x^n$ term of a formal power series $P(x)\in R[[x]]$.
\begin{lemma}[Lagrange's Implicit Function Theorem]\label{lift}
Let $R$ be a commutative ring containing $\mathbb{Q}$. Suppose $\phi(x)\in R[[x]]$ and $[x^0]\phi(x)$ is invertible in $R$. Then, there exists a unique non-zero $P(x)\in R[[x]]$ satisfying $P(x)=x\phi(P(x))$. Moreover, for any $n\geq 1$ and any $\psi\in R[[x]]$,
$$[x^n]\psi(P(x))=\frac1{n}[x^{n-1}]\psi'(x)\phi(x)^n.$$
In particular, if we set $\psi(x)=x$, then $[x^0]P(x)=0$, and for $n\geq 1$,
$$[x^n]P(x)=\frac1{n}[x^{n-1}]\phi(x)^n.$$
\end{lemma}

\begin{theorem}\label{thm:main}
We obtain formulas for $p_n$ in the following special cases.
\stepcounter{propcounter}
\begin{enumerate}[label= \upshape\textbf{\Alph{propcounter}\arabic{enumi}}]
    \item\label{main:1} If $\mathcal{U}=\{(1,s)\}$, $r_i=1$ for all $i\in J$ and $W=\{(1,-1)\}$, then
\[p_n=\sum_{i=0}^{\floor{\frac{n}{s+1}}}\frac1{n-i+1}\left([x^i]F(x)^{n-i+1}\right)\left([x^{n-(s+1)i}]\left(1+xV(1,x)\right)^{n-i+1}\right).\]
    \item\label{main:2} If $\mathcal{U}=\{(r_1,s_1)\}$, $\mathcal{V}=\{(r_2,s_2)\}$ and $W=\{(w,-1)\}$, then
\[p_n=\sum_{\substack{(ws_1+r_1)i+(ws_2+r_2)j=n\\i,j\geq0}}\frac{1}{s_1i+(s_2+1)j+1}\left([x^i]F(x)^{s_1i+(s_2+1)j+1}\right)\binom{s_1i+(s_2+1)j+1}{j}.\]
\end{enumerate}
\end{theorem}
\begin{proof}
For \ref{main:1}, substituting the given conditions into the definitions of $U(x,y),V(x,y),W(x)$, we get from Theorem \ref{thm:master} that
\[P(x)=F\left(x^{s+1}P(x)^s\right)\left(1+\sum_{i\in J}x^{s_i+1}P(x)^{s_i+1}\right).\]
By Lemma \ref{lift}, there is a unique bivariate formal power series $\overline{P}(x,z)$, viewed as a formal power series in $\mathbb{C}[[x]][[z]]$ (that is, a formal power series in $z$ with coefficients in $\mathbb{C}[[x]]$) satisfying
\[\overline{P}(x,z)=zF\left(x^{s+1}\overline{P}(x,z)^s\right)\left(1+\sum_{i\in J}x^{s_i+1}\overline{P}(x,z)^{s_i+1}\right).\]
Moreover, by Lemma \ref{lift}, we have, for $m\geq1$
\begin{align*}
[z^m]\overline{P}(x,z)&=\frac1m[z^{m-1}]F\left(x^{s+1}z^s\right)^m\left(1+\sum_{i\in J}x^{s_i+1}z^{s_i+1}\right)^m\\
&=\frac1m\sum_{\substack{sa+b=m-1\\a,b\geq0}}\left([x^a]F(x)^m\right)x^{(s+1)a}\left([x^b]\left(1+\sum_{i\in J}x^{s_i+1}\right)^m\right)x^{b}\\
&=\frac1m\sum_{\substack{sa+b=m-1\\a,b\geq0}}\left([x^a]F(x)^m\right)\left([x^b]\left(1+\sum_{i\in J}x^{s_i+1}\right)^m\right)x^{(s+1)a+b}.
\end{align*}
Therefore, for $n\geq 0$, viewing $\overline{P}(x,z)$ now as in $\mathbb{C}[[z]][[x]]$ and collecting the terms, we have
\[[x^n]\overline{P}(x,z)=\sum_{\substack{(s+1)a+b=n\\a,b\geq0}}\frac{1}{sa+b+1}\left([x^a]F(x)^{sa+b+1}\right)\left([x^b]\left(1+\sum_{i\in J}x^{s_i+1}\right)^{sa+b+1}\right)z^{sa+b+1}.\]
Since this is a finite sum for each $n$, we may substitute $z=1$ into $\overline{P}(x,z)$ to obtain
\[\overline{P}(x,1)=F\left(x^{s+1}\overline{P}(x,1)^s\right)\left(1+\sum_{i\in J}x^{s_i+1}\overline{P}(x,1)^{s_i+1}\right).\] 
Thus, $\overline{P}(x,1)=P(x)$ by Lemma \ref{lift} as they satisfy the same functional equation. Hence, 
\begin{align*}
p_n&=[x^n]P(x)=[x^n]\overline{P}(x,1)\\
&=\sum_{\substack{(s+1)a+b=n\\a,b\geq0}}\frac1{sa+b+1}\left([x^a]F(x)^{sa+b+1}\right)\left([x^b]\left(1+\sum_{i\in J}x^{s_i+1}\right)^{sa+b+1}\right),\\
&=\sum_{i=0}^{\floor{\frac{n}{s+1}}}\frac1{n-i+1}\left([x^i]F(x)^{n-i+1}\right)\left([x^{n-(s+1)i}]\left(1+xV(1,x)\right)^{n-i+1}\right),
\end{align*}
as required.

In the case of \ref{main:2}, we have by Theorem \ref{thm:master} that
\[P(x)=F\left(x^{ws_1+r_1}P(x)^{s_1}\right)\left(1+x^{ws_2+r_2}P(x)^{s_2+1}\right).\]
Like above, let $\overline{P}(x,z)$ be the unique formal power series in $\mathbb{C}[[x]][[z]]$ satisfying
\[\overline{P}(x,z)=zF\left(x^{ws_1+r_1}\overline{P}(x,z)^{s_1}\right)\left(1+x^{ws_2+r_2}\overline{P}(x,z)^{s_2+1}\right).\]
By Lemma \ref{lift}, we have, for $m\geq1$,
\begin{align*}
[z^m]\overline{P}(x,z)&=\frac1m[z^{m-1}]F\left(x^{ws_1+r_1}z^{s_1}\right)^m\left(1+x^{ws_2+r_2}z^{s_2+1}\right)^m\\
&=\frac1m\sum_{\substack{s_1i+(s_2+1)j=m-1\\i,j\geq0}}\left([x^i]F(x)^m\right)x^{(ws_1+r_1)i}\binom{m}{j}x^{(ws_2+r_2)j}\\
&=\frac1m\sum_{\substack{s_1i+(s_2+1)j=m-1\\i,j\geq0}}\left([x^i]F(x)^m\right)\binom{m}{j}x^{(ws_1+r_1)i+(ws_2+r_2)j}.
\end{align*}
Therefore, viewing $\overline{P}(x,z)$ now as in $\mathbb{C}[[z]][[x]]$, for $n\geq 0$, $[x^n]\overline{P}(x,z)$ is equal to
\[\sum_{\substack{(ws_1+r_1)i+(ws_2+r_2)j=n\\i,j\geq0}}\frac{1}{s_1i+s_2j+j+1}\left([x^i]F(x)^{s_1i+s_2j+j+1}\right)\binom{s_1i+s_2j+j+1}{j}z^{s_1i+s_2j+j+1}.\]
Since this is a finite sum for each $n$, we may substitute $z=1$ into $\overline{P}(x,z)$ to obtain
\[\overline{P}(x,1)=F\left(x^{ws_1+r_1}\overline{P}(x,1)^{s_1}\right)\left(1+x^{ws_2+r_2}\overline{P}(x,1)^{s_2+1}\right).\] 
Thus, $\overline{P}(x,1)=P(x)$ by Lemma \ref{lift} as they satisfy the same functional equation. Hence, 
\begin{align*}
p_n&=[x^n]P(x)=[x^n]\overline{P}(x,1)\\
&=\sum_{\substack{(ws_1+r_1)i+(ws_2+r_2)j=n\\i,j\geq0}}\frac{1}{s_1i+(s_2+1)j+1}\left([x^i]F(x)^{s_1i+(s_2+1)j+1}\right)\binom{s_1i+(s_2+1)j+1}{j},
\end{align*}
as required.
\end{proof}

\section{Catalan paths}\label{sec:catalan}
For $k\in\mathbb{Z}_+$, the set $\mathcal{C}_n^{(k)}$ of $k$-Catalan paths of size $n$, is the set of lattice paths from $(0,0)$ to $((k+1)n,0)$ with step set $\{(1,k),(1,-1)\}$ that never go below the $x$-axis. It is well-known that $|\mathcal{C}_n^{(k)}|=\frac1{kn+1}\binom{(k+1)n}{n}$. By applying appropriate geometric transformations, it is easy to see that the set $\mathcal{C}_n^{(k)}$ is in bijection with the following sets of lattice paths. 
\begin{itemize}
    \item Lattice paths from $(0,0)$ to $(n,kn)$ with step set $\{(0,1),(1,0)\}$ that never go below the line $y=kx$.
    \item Lattice paths from $(0,0)$ to $((k+1)n,(k-1)n)$ with step set $\{(1,1),(1,-1)\}$ that never go below the line $y=\frac{k-1}{k+1}x$.
\end{itemize}
This observation will motivate our generalisations of Schr\"{o}der paths in Section \ref{sec:schroeder} and Motzkin paths in Section \ref{sec:motzkin}.

For each $C\in\mathcal{C}_n^{(k)}$, let $\mathbf{u}(C)$ be the vector recording the length of each maximum block of consecutive $(1,k)$ steps in $C$. In the language of Section \ref{sec:main}, we have that if $\mathcal{U}=\{(1,k)\}, \mathcal{V}=\emptyset, \mathcal{W}=\{(1,-1)\}$, then $\mathcal{C}_n^{(k)}=\mathcal{P}_{(k+1)n}$. As such, we can apply Theorem \ref{thm:master} and Theorem \ref{thm:main} to obtain the following general weighted enumeration result for $\mathcal{C}_n^{(k)}$. We remark that this result can be obtain more directly in a slightly simpler manner, but we opted to prove more general results in Section \ref{sec:main}, and obtain the following as a corollary instead. 

\begin{theorem}\label{thm:catalan}
For any function $f:\mathbb{Z}_+\to\mathbb{C}$, let $F(x)=1+\sum_{\ell=1}^\infty f(\ell)x^\ell$. Then for every $k\in\mathbb{Z}_+$ and $n\in\mathbb{Z}_{\geq0}$, we have \[c_n:=\sum_{C\in\mathcal{C}_n^{(k)}}\prod_{i=1}^{|\mathbf{u}(C)|}f(\mathbf{u}(C)_i)=\frac{1}{kn+1}[x^n]F(x)^{kn+1},\]
and the generating function $C(x)=\sum_{n=0}^\infty c_nx^n$ satisfies the functional equation
\[C(x)=F(xC(x)^k).\]
\end{theorem}
\begin{proof}
Apply Theorem \ref{thm:master} with $U(x,y)=xy^k, V(x,y)=0$ and $W(x)=x$, we get that $P(x)=\sum_{n=0}^\infty p_nx^n$ satisfies
\[P(x)=F(x^{k+1}P(x)^k).\]
Noting that $\mathcal{P}_n\not=\emptyset$ only if $k+1$ divides $n$, and that $c_n=p_{(k+1)n}$ for all $n\in\mathbb{Z}_{\geq0}$, we see that $P(x)=C(x^{k+1})$. Hence, $C(x^{k+1})=F(x^{k+1}C(x^{k+1})^k)$, from which it follows that $C(x)=F(xC(x)^k)$.

To extract the coefficients of $C(x)$, one could proceed similarly to the proof of \ref{main:1} in Theorem \ref{thm:main}, but we instead apply \ref{main:1} directly to $P(x)$, which gives 
\begin{align*}
c_n=\sum_{C\in\mathcal{C}_n^{(k)}}\prod_{i=1}^{|\mathbf{u}(C)|}f(\mathbf{u}(C)_i)&=p_{(k+1)n}\\
&=\sum_{i=0}^{n}\frac1{(k+1)n-i+1}\left([x^i]F(x)^{(k+1)n-i+1}\right)\left([x^{(k+1)(n-i)}]1\right)\\
&=\frac1{kn+1}[x^n]F(x)^{kn+1},
\end{align*}
as required, where the last equality follows as only the $i=n$ term in the sum is non-zero.
\end{proof}

As a corollary to Theorem \ref{thm:catalan}, we can obtain explicit weighted enumeration formulas with respect to several notable weight functions. The functional equation satisfied by the generating functions of these sequences can also be obtained easily from Theorem \ref{thm:catalan}. 
\begin{corollary}\label{cor:cat}
For every $k\in\mathbb{Z}_+$ and $n\in\mathbb{Z}_{\geq0}$, we have
\stepcounter{propcounter}
\begin{enumerate}[label= \upshape\textbf{\Alph{propcounter}\arabic{enumi}}]
    \item\label{cor:cat1}\[|\mathcal{C}_n^{(k)}|=\frac1{kn+1}\binom{(k+1)n}{n},\]
and its generating function satisfies
\[C(x)=\frac1{1-xC(x)^k}.\]
    \item\label{cor:cat2}\[\sum_{C\in\mathcal{C}_n^{(k)}}\prod_{i=1}^{|\mathbf{u}(C)|}m\mathbf{u}(C)_i=\frac1{kn+1}\sum_{i=1}^{n}\binom{kn+1}{i}\binom{n+i-1}{n-i}m^i,\]
and its generating function satisfies
\[C(x)=1+\frac{mxC(x)^k}{(1-xC(x)^k)^2}.\]
     \item\label{cor:cat3}\[\sum_{C\in\mathcal{C}_n^{(k)}}\prod_{i=1}^{|\mathbf{u}(C)|}(1+m\mathbf{u}(C)_i)=\frac1{kn+1}\sum_{i=0}^n\binom{kn+1}{i}\binom{(2k+1)n-i+1}{n-i}(m-1)^i,\]
and its generating function satisfies
\[C(x)=\frac{1+(m-1)xC(x)^k}{(1-xC(x)^k)^2}.\]
\end{enumerate}
\end{corollary}
\begin{proof}
For \ref{cor:cat1}, apply Theorem \ref{thm:catalan} with $f(\ell)=1$, we get $F(x)=\frac1{1-x}$ and so
\[|\mathcal{C}_n^{(k)}|=\sum_{C\in\mathcal{C}_n^{(k)}}\prod_{i=1}^{|\mathbf{u}(C)|}1
=\frac{1}{kn+1}[x^n]\frac{1}{(1-x)^{kn+1}}=\frac1{kn+1}\binom{(k+1)n}{n}.\]

For \ref{cor:cat2}, apply Theorem \ref{thm:catalan} with $f(\ell)=m\ell$, we get $F(x)=1+\frac{mx}{(1-x)^2}$ and so
\[\sum_{C\in\mathcal{C}_n^{(k)}}\prod_{i=1}^{|\mathbf{u}(C)|}m\mathbf{u}(C)_i=\frac{1}{kn+1}[x^n]\left(1+\frac{mx}{(1-x)^2}\right)^{kn+1}=\frac1{kn+1}\sum_{i=1}^{n}\binom{kn+1}{i}\binom{n+i-1}{n-i}m^i.\]

For \ref{cor:cat3}, apply Theorem \ref{thm:catalan} with $f(\ell)=1+m\ell$, we get $F(x)=\frac 1{1-x}+\frac{mx}{(1-x)^2}=\frac{1+(m-1)x}{(1-x)^2}$ and so
\begin{align*}
\sum_{C\in\mathcal{C}_n^{(k)}}\prod_{i=1}^{|\mathbf{u}(C)|}(1+m\mathbf{u}(C)_i)&=\frac{1}{kn+1}[x^n]\frac{(1+(m-1)x)^{kn+1}}{(1-x)^{2kn+2}}\\
&=\frac1{kn+1}\sum_{i=0}^n\binom{kn+1}{i}\binom{(2k+1)n-i+1}{n-i}(m-1)^i.
\end{align*}

The functional equations that the generating functions for the sequences in \ref{cor:cat1}-\ref{cor:cat3} satisfy follow by direct substitutions of the corresponding $F(x)$ above into the second part of Theorem \ref{thm:catalan}.
\end{proof}

To end this section, we discuss the connections between these weighted enumerations and other enumeration problems. 

The sequences given by \ref{cor:cat1} enumerate the $k$-Catalan paths and are the well-known Fuss-Catalan numbers. They also enumerate many other combinatorial objects, such as the $(k+1)$-ary trees on $n$ vertices. When $k=1,2,3,4,5,6$, these are OEIS sequences A000108, A001764, A002293, A002294, A002295 and A002296, respectively. 

By \cite{Y}, when $k=m=1$, the sequence given by \ref{cor:cat2}, which is A109081 in OEIS, counts parking functions of size $n$ whose parking outcomes avoid the pattern 123. The following table shows several other matches between the sequences given by \ref{cor:cat2} and OEIS entries. In all of these instances, our results provide first combinatorial interpretations for the corresponding OEIS entries, which were all defined purely algebraically. 

\begin{center}
\begin{tabular}{|c|c|c|c|c|c|}
   \hline$(m,k)$  & $(1,2)$ & $(1,3)$ & $(2,1)$ & $(4,2)$ & $(4,3)$ \\\hline
   OEIS entry  & A367237 & A367280 & A369208 & A371676 & A371678 \\\hline
\end{tabular}    
\end{center}
Interestingly, when $m=4$ and $k=1$, we get OEIS sequence A032349, which reappears later in \ref{cor:schr4} and is related to generalised Schr\"{o}der paths. We delay the discussion till then.

When $m=1$ in \ref{cor:cat3}, the sums reduces to \[\sum_{C\in\mathcal{C}_n^{(k)}}\prod_{i=1}^{|\mathbf{u}(C)|}(1+\mathbf{u}(C)_i)=\frac1{kn+1}\binom{(2k+1)n+1}{n},\]
since only the $i=0$ term is non-zero. By \cite{S}, this is the number of $(2k)_1$-Dyck paths of length $(2k+1)n$, or equivalently the number of pairs of $2k$-Catalan paths of size $n$. This is also the number of $(k+1)$-ary hybrid trees with $n$ internal vertices avoiding a certain pattern by \cite{YJ2}. For $k=1,2,3,4$, the OEIS entries are A006013, A118969, A233832 and A234505, respectively. Moreover, by using these results and similar decompositions as in the proof of Theorem \ref{thm:main}, we can evaluate the following sums
\[\sum_{C\in\mathcal{C}_n^{(k)}}\prod_{i=2}^{|\mathbf{u}(C)|}(1+\mathbf{u}(C)_i)=\frac1{2kn+1}\binom{(2k+1)n}{n},\]
where, compared to above, the first terms in the products are omitted. These turn out to be equal to the number of hyposylvester classes of $k$-parking functions as defined by Novelli and Thibon in \cite{NT}. 

When $m=2$ and $k=1$ in \ref{cor:cat3}, this is OEIS sequence A003169, and by \cite{C} it counts the number of 2-line arrays. When $m=k=2$, this provides a combinatorial interpretation of OEIS sequence A278745. 

When $k=1$ in \ref{cor:cat3}, these sequences can similarly be used along with suitable decompositions as in the proof of Theorem \ref{thm:main} to obtain the evaluations
\[\sum_{C\in\mathcal{C}_n}\prod_{i=2}^{|\mathbf{u}(C)|}(1+m\mathbf{u}(C)_i)=\frac1n\sum_{i=0}^{n-1}\binom{n}{i}\binom{3n-i}{2n+1}(m-1)^i,\]
which are shown in \cite{Y} to be equal to the number of hyposylvester classes of $m$-multiparking functions as defined by Novelli and Thibon in \cite{NT}. When $k=m=1$, the sequence is used in \cite{Y} to show
\[\sum_{C\in\mathcal{C}_n}\prod_{i=1}^{|\mathbf{u}(C)|-1}(1+\mathbf{u}(C)_i)=\frac{\binom{3n+1}{n}}{n+1}-\sum_{i=0}^{n-1}\frac{\binom{3n-3i+1}{n-i}}{2^{i+1}(n-i+1)},\]
which is equal to the number of parking functions of size $n$ that avoids both patterns 312 and 321 under a different notion used by Adeniran and Pudwell in \cite{AP}.

Finally, we note that if we use the non-linear weight function $f(\ell)=\ell!$, then by \cite{Y} the sum over $\mathcal{C}_n$ enumerates parking functions of size $n$ whose parking outcomes avoid the pattern 213, and this is sequence A088368 in OEIS.

\section{Schr\"{o}der paths}\label{sec:schroeder}
The set $\mathcal{S}_n$ of Schr\"{o}der paths of size $n$, is the set of lattice paths from $(0,0)$ to $(2n,0)$ with step set $\{(1,1),(2,0),(1,-1)\}$ that never go below the $x$-axis. Alternatively, this is also commonly defined as the set of lattice paths from $(0,0)$ to $(n,n)$ with step set $\{(0,1),(1,1),(1,0)\}$ that never go below $y=x$, and it is easy to find a bijection between these two sets of lattice paths. 

There are many generalisations of Schr\"{o}der paths in the literature, but in this section we use the following generalisation introduced recently by Yang and Jiang \cite{YJ}. For $k\in\mathbb{Z}_+$, the set $\mathcal{S}_n^{(k)}$ of $k$-Schr\"{o}der paths of size $n$ is the set of lattice paths from $(0,0)$ to $((k+1)n,0)$ with step set $\{(1,k),(2,k-1),(1,-1)\}$ that never go below the $x$-axis. This resembles the way $\mathcal{C}_n^{(k)}$ generalises $\mathcal{C}_n$ in Section \ref{sec:catalan}, because by appropriate geometric transformations, $\mathcal{S}_n^{(k)}$ is in bijection with the following sets of lattice paths.
\begin{itemize}
    \item Lattice paths from $(0,0)$ to $(n,kn)$ with step set $\{(0,1),(1,1),(1,0)\}$ that never go below the line $y=kx$.
    \item Lattice paths from $(0,0)$ to $((k+1)n,(k-1)n)$ with step set $\{(1,1),(2,0),(1,-1)\}$ that never go below the line $y=\frac{k-1}{k+1}x$.
\end{itemize}

For each $S\in\mathcal{S}_n^{(k)}$, let $\mathbf{u}(S)$ be the vector recording the length of each maximum block of consecutive $(1,k)$ steps in $S$. In the language of Section \ref{sec:main}, we have that if $\mathcal{U}=\{(1,k)\}, \mathcal{V}=\{(2,k-1)\}, \mathcal{W}=\{(1,-1)\}$, then $\mathcal{S}_n^{(k)}=\mathcal{P}_{(k+1)n}$. As such, we can apply Theorem \ref{thm:master} and Theorem \ref{thm:main} to obtain the following weighted enumeration result for $\mathcal{S}_n^{(k)}$.
\begin{theorem}\label{thm:schroeder}
For any function $f:\mathbb{Z}_+\to\mathbb{C}$, let $F(x)=1+\sum_{\ell=1}^\infty f(\ell)x^\ell$. Then for every $k\in\mathbb{Z}_+$ and $n\in\mathbb{Z}_{\geq0}$, we have
\[s_n:=\sum_{S\in\mathcal{S}_n^{(k)}}\prod_{i=1}^{|\mathbf{u}(S)|}f(\mathbf{u}(S)_i)=\frac{1}{kn+1}[x^n](1+x)^{kn+1}F(x)^{kn+1},\]
and the generating function $S(x)=\sum_{n=0}^\infty s_nx^n$ satisfies the functional equation
\[S(x)=F(xS(x)^k)(1+xS(x)^k).\]
\end{theorem}
\begin{proof}
Apply Theorem \ref{thm:master} with $U(x,y)=xy^k, V(x,y)=x^2y^{k-1}$ and $W(x)=x$, we get that $P(x)=\sum_{n=0}^\infty p_nx^n$ satisfies
\[P(x)=F(x^{k+1}P(x)^k)(1+x^{k+1}P(x)^k).\]
Noting that $\mathcal{P}_n\not=\emptyset$ only if $k+1$ divides $n$, and that $s_n=p_{(k+1)n}$ for all $n\in\mathbb{Z}_{\geq0}$, we see that $P(x)=S(x^{k+1})$. Hence, $S(x^{k+1})=F(x^{k+1}S(x^{k+1})^k)(1+x^{k+1}S(x^{k+1})^k)$, from which it follows that $S(x)=F(xS(x)^k)(1+xS(x)^k)$.

To get a formula for $s_n$, by \ref{main:2} of Theorem \ref{thm:main}, we have
\begin{align*}
s_n=\sum_{S\in\mathcal{S}_n^{(k)}}\prod_{i=1}^{|\mathbf{u}(S)|}f(\mathbf{u}(S)_i)&=p_{(k+1)n}\\
&=\sum_{\substack{(k+1)i+(k+1)j=(k+1)n\\i,j\geq0}}\frac{1}{ki+kj+1}\left([x^i]F(x)^{ki+kj+1}\right)\binom{ki+kj+1}{j}\\
&=\sum_{\substack{i+j=n\\i,j\geq0}}\frac{1}{kn+1}\left([x^i]F(x)^{kn+1}\right)\binom{kn+1}{j}\\
&=\frac{1}{kn+1}[x^n](1+x)^{kn+1}F(x)^{kn+1},
\end{align*}
as required.
\end{proof}

As a corollary, we get the following explicit formulas. 
\begin{corollary}\label{cor:schr}
For every $k\in\mathbb{Z}_+$ and $n\in\mathbb{Z}_{\geq0}$, we have
\stepcounter{propcounter}
\begin{enumerate}[label= \upshape\textbf{\Alph{propcounter}\arabic{enumi}}]
    \item\label{cor:schr1}\[|\mathcal{S}_n^{(k)}|=\frac1{kn+1}\sum_{i=0}^n\binom{kn+1}{n-i}\binom{kn+i}{i},\]
and its generating function satisfies
\[S(x)=\frac{1+xS(x)^k}{1-xS(x)^k}.\]
    \item\label{cor:schr2}\[\sum_{S\in\mathcal{S}_n^{(k)}}\prod_{i=1}^{|\mathbf{u}(S)|}\mathbf{u}(S)_i=\frac1{kn+1}\sum_{i=0}^{\floor{\frac n3}}\binom{kn+1}{i}\binom{(2k+1)n-3i+1}{n-3i},\]
and its generating function satisfies
\[S(x)=\frac{1+x^3S(x)^{3k}}{(1-xS(x)^k)^2}.\]
    \item\label{cor:schr3}\[\sum_{S\in\mathcal{S}_n^{(k)}}\prod_{i=1}^{|\mathbf{u}(S)|}(1+\mathbf{u}(S)_i)=\frac1{kn+1}\sum_{i=0}^n\binom{kn+1}{i}\binom{(2k+1)n-i+1}{n-i},\]
and its generating function satisfies
\[S(x)=\frac{1+xS(x)^k}{(1-xS(x)^k)^2}.\]
    \item\label{cor:schr4}\[\sum_{S\in\mathcal{S}_n^{(k)}}\prod_{i=1}^{|\mathbf{u}(S)|}(1+2\mathbf{u}(S)_i)=\frac1{kn+1}\sum_{i=0}^n\binom{2kn+2}{i}\binom{(2k+1)n-i+1}{n-i},\]
and its generating function satisfies
\[S(x)=\frac{(1+xS(x)^k)^2}{(1-xS(x)^k)^2}.\]
\end{enumerate}
\end{corollary}
\begin{proof}
For \ref{cor:schr1}, apply Theorem \ref{thm:schroeder} with $f(\ell)=1$, we get $F(x)=\frac1{1-x}$ and so
\begin{align*}
|\mathcal{S}_n^{(k)}|&=\sum_{S\in\mathcal{S}_n^{(k)}}\prod_{i=1}^{|\mathbf{u}(S)|}f(\mathbf{u}(S)_i)
=\frac{1}{kn+1}[x^n]\frac{(1+x)^{kn+1}}{(1-x)^{kn+1}}=\frac1{kn+1}\sum_{i=0}^n\binom{kn+1}{n-i}\binom{kn+i}{i}.
\end{align*}

For \ref{cor:schr2}, apply Theorem \ref{thm:schroeder} with $f(\ell)=\ell$, we get $F(x)=1+\frac{x}{(1-x)^2}=\frac{1-x+x^2}{(1-x)^2}$ and so
\begin{align*}
\sum_{S\in\mathcal{S}_n^{(k)}}\prod_{i=1}^{|\mathbf{u}(S)|}\mathbf{u}(S)_i&=\frac{1}{kn+1}[x^n](1+x)^{kn+1}\frac{(1-x+x^2)^{kn+1}}{(1-x)^{2kn+2}}\\
&=\frac{1}{kn+1}[x^n]\frac{(1+x^3)^{kn+1}}{(1-x)^{2kn+2}}=\frac1{kn+1}\sum_{i=0}^{\floor{\frac n3}}\binom{kn+1}{i}\binom{(2k+1)n-3i+1}{n-3i}.
\end{align*}

For \ref{cor:schr3}, apply Theorem \ref{thm:schroeder} with $f(\ell)=1+\ell$, we get $F(x)=\frac 1{1-x}+\frac{x}{(1-x)^2}=\frac{1}{(1-x)^2}$ and so
\begin{align*}
\sum_{S\in\mathcal{S}_n^{(k)}}\prod_{i=1}^{|\mathbf{u}(S)|}(1+\mathbf{u}(S)_i)&=\frac{1}{kn+1}[x^n]\frac{(1+x)^{kn+1}}{(1-x)^{2kn+2}}=\frac1{kn+1}\sum_{i=0}^n\binom{kn+1}{i}\binom{(2k+1)n-i+1}{n-i}.
\end{align*}

For \ref{cor:schr4}, apply Theorem \ref{thm:schroeder} with $f(\ell)=1+2\ell$, we get $F(x)=\frac 1{1-x}+\frac{2x}{(1-x)^2}=\frac{1+x}{(1-x)^2}$ and so
\begin{align*}
\sum_{S\in\mathcal{S}_n^{(k)}}\prod_{i=1}^{|\mathbf{u}(S)|}(1+2\mathbf{u}(S)_i)&=\frac{1}{kn+1}[x^n]\frac{(1+x)^{2kn+2}}{(1-x)^{2kn+2}}=\frac1{kn+1}\sum_{i=0}^n\binom{2kn+2}{i}\binom{(2k+1)n-i+1}{n-i}.
\end{align*}

The functional equations that the generating functions for the sequences in \ref{cor:schr1}-\ref{cor:schr4} satisfy follow by direct substitutions of the corresponding $F(x)$ above into the second part of Theorem \ref{thm:schroeder}.
\end{proof}

We again end the section by discussing the connections between the sequences in Corollary \ref{cor:schr} and other combinatorial objects. The sequences in \ref{cor:schr1} are the large $k$-Schr\"{o}der numbers as defined by Yang and Jiang in \cite{YJ} (although they called them large $(k+1)$-Schr\"{o}der numbers). The OEIS entries when $k=1,2,3,4$ are A006318, A027307, A144097 and A260332, respectively, with various other interpretations listed. 

The $k=1$ case of \ref{cor:schr2} gives the OEIS sequence A369265, and thus provides a combinatorial interpretation of that sequence which comes from a series reversion operation.

The sequences in \ref{cor:schr3} coincides with the sequences in \ref{cor:cat3} when $m=2$ which was discussed above, as their generating functions satisfy the same functional equation. It would be interesting to find a combinatorial proof of this fact.

For \ref{cor:schr4}, when $k=1$ this coincides with A032349, which counts the number of $2$-Schr\"{o}der paths of size $n$ that start with a $(2,1)$-step. More generally, the sequences in \ref{cor:schr4} coincide with the $m=4$ sequences in \ref{cor:cat2}. Both of these facts can be proved straightforwardly using generating functions, but it would be nice to have more combinatorial proofs. Moreover, as mentioned earlier, the $k=2$ and $3$ cases of \ref{cor:schr4} provide combinatorial interpretations for OEIS sequences A371676 and A371678.


\section{Motzkin paths}\label{sec:motzkin}
The set $\mathcal{M}_n$ of Motzkin paths of size $n$, is the set of lattice paths from $(0,0)$ to $(n,0)$ with step set $\{(1,1),(1,0),(1,-1)\}$ that never go below the $x$-axis. Alternatively, it is also commonly defined as the set of lattice paths from $(0,0)$ to $(n,n)$ with step set $\{(0,2),(1,1),(2,0)\}$ that never go below $y=x$, and it is easy to find a bijection between these two sets of lattice paths. 

There are also many generalisations of Motzkin paths in the literature, including one recently proposed by Yang, Zhang and Yang \cite{YZY}, whose definition of $k$-Motzkin paths coincides with our definition of $(2k-1)$-Motzkin paths below. As far as we know, our definition of $k$-Motzkin paths when $k$ is even is new. 

For odd $k\in\mathbb{Z}_+$, the set $\mathcal{M}_n^{(k)}$ of $k$-Motzkin paths of size $n$ is defined to be the set of lattice paths from $(0,0)$ to $(\frac{k+1}2n,0)$ with step set $\{(1,k),(1,\frac{k-1}2),(1,-1)\}$ that never go below the $x$-axis. This resembles the generalisations in Section \ref{sec:catalan} and Section \ref{sec:schroeder}, because by appropriate geometric transformations, $\mathcal{M}_n^{(k)}$ is in bijection with the following sets of lattice paths.
\begin{itemize}
    \item Lattice paths from $(0,0)$ to $(n,kn)$ with step set $\{(0,2),(1,1),(2,0)\}$ that never go below the line $y=kx$.
    \item Lattice paths from $(0,0)$ to $(\frac{k+1}2n,\frac{k-1}2n)$ with step set $\{(1,1),(1,0),(1,-1)\}$ that never go below the line $y=\frac{k-1}{k+1}x$.
\end{itemize}

For even $k\in\mathbb{Z}_+$, we define the set $\mathcal{M}_n^{(k)}$ of $k$-Motzkin paths of size $n$ to be the set of lattice paths from $(0,0)$ to $((k+1)n,0)$ with step set $\{(1,2k),(1,k-1),(1,-2)\}$ that never go below the $x$-axis. This again resembles the generalisations in Section \ref{sec:catalan} and Section \ref{sec:schroeder}, because by appropriate geometric transformations, $\mathcal{M}_n^{(k)}$ is in bijection with the following sets of lattice paths.
\begin{itemize}
    \item Lattice paths from $(0,0)$ to $(2n,2kn)$ with step set $\{(0,2),(1,1),(2,0)\}$ that never go below the line $y=kx$.
    \item Lattice paths from $(0,0)$ to $((k+1)n,(k-1)n)$ with step set $\{(1,1),(1,0),(1,-1)\}$ that never go below the line $y=\frac{k-1}{k+1}x$.
\end{itemize}

However, unlike the generalisations in Section \ref{sec:catalan} and Section \ref{sec:schroeder}, the parity of $k$ plays quite an important role in our definition of $\mathcal{M}_n^{(k)}$, with even values of $k$ presenting some complications, which is perhaps why the authors of \cite{YZY} defined their generalisation only for the odd cases. For example, when $k$ is even, the vertical size of each step is doubled compared to the case when $k$ is odd, in order to avoid fractional coordinates. Further, one can easily show that when $k$ is odd, $k$-Motzkin paths can return to the $x$-axis precisely at points whose $x$-coordinates are divisible by $\frac{k+1}2$, but when $k$ is even, they can return to the $x$-axis precisely when the $x$-coordinates are divisible by $k+1$. This explains why our definitions of the size of $k$-Motzkin paths differ depending on the parity of $k$. We will deal with the odd and even cases separately in Section~\ref{odd} and Section~\ref{even} below, respectively. As will be seen, (weighted) enumerations of $\mathcal{M}_n^{(k)}$ are more difficult when $k$ is even. 

\subsection{$k$-Motzkin paths with $k$ odd}\label{odd}
We first deal with the case when $k$ is odd. For each $M\in\mathcal{M}_n^{(k)}$, let $\mathbf{u}(M)$ be the vector recording the length of each maximum block of consecutive $(1,k)$ steps in $M$. In the language of Section \ref{sec:main}, we have that if $\mathcal{U}=\{(1,k)\}, \mathcal{V}=\{(1,\frac{k-1}2)\}, \mathcal{W}=\{(1,-1)\}$, then $\mathcal{M}_n^{(k)}$ is exactly $\mathcal{P}_{\frac{k+1}2n}$. As such, we can apply Theorem \ref{thm:master} and Theorem \ref{thm:main} to obtain the following weighted enumeration results for $\mathcal{M}_n^{(k)}$ when $k$ is odd.

\begin{theorem}\label{thm:motzkin}
For any function $f:\mathbb{Z}_+\to\mathbb{C}$, let $F(x)=1+\sum_{\ell=1}^\infty f(\ell)x^\ell$. Then for every $n\in\mathbb{Z}_{\geq0}$ and every odd $k\in\mathbb{Z}_+$, we have
\[m_n:=\sum_{M\in\mathcal{M}_n^{(k)}}\prod_{i=1}^{|\mathbf{u}(M)|}f(\mathbf{u}(M)_i)=\sum_{i=0}^{\floor{\frac n2}}\frac{1}{\frac{k+1}2n-i+1}\binom{\frac{k+1}2n-i+1}{n-2i}[x^i]F(x)^{\frac{k+1}2n-i+1},\]
and the generating function $M(x)=\sum_{n=0}^\infty m_nx^n$ satisfies the functional equation
\[M(x)=F(x^2M(x)^k)(1+xM(x)^{\frac{k+1}2}).\]
\end{theorem}
\begin{proof}
Apply Theorem \ref{thm:master} with $U(x,y)=xy^k, V(x,y)=xy^{\frac{k-1}2}$ and $W(x)=x$, we get that $P(x)=\sum_{n=0}^\infty p_nx^n$ satisfies
\[P(x)=F(x^{k+1}P(x)^k)(1+x^{\frac{k+1}2}P(x)^{\frac{k+1}2}).\]
Noting that $\mathcal{P}_n\not=\emptyset$ only if $\frac{k+1}2$ divides $n$, and that $m_n=p_{\frac{(k+1)}2n}$ for all $n\in\mathbb{Z}_{\geq0}$, we see that $P(x)=M(x^{\frac{k+1}2})$. Hence, $M(x^{\frac{k+1}2})=F(x^{k+1}M(x^{\frac{k+1}2})^k)(1+x^{\frac{k+1}2}M(x^{\frac{k+1}2})^{\frac{k+1}2})$, from which it follows that $M(x)=F(x^2M(x)^k)(1+xM(x)^{\frac{k+1}2})$.

Moreover, applying \ref{main:1} of Theorem \ref{thm:main}, we have
\begin{align*}
&\phantom{==}\sum_{M\in\mathcal{M}_n^{(k)}}\prod_{i=1}^{|\mathbf{u}(M)|}f(\mathbf{u}(M)_i)=p_{\frac{k+1}2n}\\
&=\sum_{i=0}^{\floor{\frac n2}}\frac{1}{\frac{k+1}2n-i+1}\left([x^i]F(x)^{\frac{k+1}2n-i+1}\right)\left([x^{\frac{k+1}2n-(k+1)i}]\left(1+x^{\frac{k+1}2}\right)^{\frac{k+1}2n-i+1}\right)\\
&=\sum_{i=0}^{\floor{\frac n2}}\frac{1}{\frac{k+1}2n-i+1}\binom{\frac{k+1}2n-i+1}{n-2i}[x^i]F(x)^{\frac{k+1}2n-i+1},
\end{align*}
as required.
\end{proof}

As a corollary, we have the following explicit formulas. 
\begin{corollary}\label{cor:motz}
For every $n\in\mathbb{Z}_{\geq0}$ and every odd $k\in\mathbb{Z}_+$, we have
\stepcounter{propcounter}
\begin{enumerate}[label= \upshape\textbf{\Alph{propcounter}\arabic{enumi}}]
    \item\label{cor:motz1}\[|\mathcal{M}_n^{(k)}|=\sum_{i=0}^{\floor{\frac n2}}\frac{1}{\frac{k+1}2n-i+1}\binom{\frac{k+1}2n-i+1}{n-2i}\binom{\frac{k+1}2n}{i},\]
and its generating function satisfies
\[M(x)=\frac{1+xM(x)^{\frac{k+1}2}}{1-x^2M(x)^k}.\]
    \item\label{cor:motz2}\[\sum_{M\in\mathcal{M}_n^{(k)}}\prod_{i=1}^{|\mathbf{u}(M)|}(1+\mathbf{u}(M)_i)=\sum_{i=0}^{\floor{\frac n2}}\frac{1}{\frac{k+1}2n-i+1}\binom{\frac{k+1}2n-i+1}{n-2i}\binom{(k+1)n-i+1}{i},\]
and its generating function satisfies
\[M(x)=\frac{1+xM(x)^{\frac{k+1}2}}{(1-x^2M(x)^k)^2}.\]
\end{enumerate}
\end{corollary}
\begin{proof}
For \ref{cor:motz1}, apply Theorem \ref{thm:motzkin} with $f(\ell)=1$, we get $F(x)=\frac1{1-x}$, and so
\begin{align*}
|\mathcal{M}_n^{(k)}|&=\sum_{i=0}^{\floor{\frac n2}}\frac{1}{\frac{k+1}2n-i+1}\binom{\frac{k+1}2n-i+1}{n-2i}[x^i]\frac1{(1-x)^{\frac{k+1}2n-i+1}}\\
&=\sum_{i=0}^{\floor{\frac n2}}\frac{1}{\frac{k+1}2n-i+1}\binom{\frac{k+1}2n-i+1}{n-2i}\binom{\frac{k+1}2n}{i},
\end{align*}

For \ref{cor:motz2}, apply Theorem \ref{thm:motzkin} with $f(\ell)=1+\ell$, we get $F(x)=\frac 1{1-x}+\frac{x}{(1-x)^2}=\frac{1}{(1-x)^2}$ and so
\begin{align*}
\sum_{M\in\mathcal{M}_n^{(k)}}\prod_{i=1}^{|\mathbf{u}(M)|}(1+\mathbf{u}(M)_i)&=\sum_{i=0}^{\floor{\frac n2}}\frac{1}{\frac{k+1}2n-i+1}\binom{\frac{k+1}2n-i+1}{n-2i}[x^i]\frac1{(1-x)^{(k+1)n-2i+2}}\\
&=\sum_{i=0}^{\floor{\frac n2}}\frac{1}{\frac{k+1}2n-i+1}\binom{\frac{k+1}2n-i+1}{n-2i}\binom{(k+1)n-i+1}{i}.
\end{align*}

The functional equations that the generating functions for the sequences in \ref{cor:motz1} and \ref{cor:motz2} satisfy follow by direct substitution of the corresponding $F(x)$ into the second part of Theorem \ref{thm:motzkin}.
\end{proof}

When $k=1$ in \ref{cor:motz1}, the sequence, which is A001006 on OEIS, is the well-known sequence of Motzkin numbers with many combinatorial interpretations. For general odd $k$, as mentioned before, the sequences in \ref{cor:motz1} count $\frac{k+1}2$-Motzkin paths as defined by Yang, Zhang and Yang \cite{YZY}. Furthermore, when $k=3$, this gives OEIS sequence A006605, which among other things counts the number of ternary trees that avoids a certain pattern by \cite{GPPT}. For $k=5,7,9,11$, the corresponding sequences give combinatorial meanings to OEIS sequences A255673, A365183, A365189 and A265033, respectively.

As far as we can tell, the sequences given by \ref{cor:motz2}, or by other weight functions, are not in OEIS. Even the easy-looking case when the weight function is $f(\ell)=\ell$ does not result in sequences with nice formulas, so it is omitted from the corollary above.

\subsection{$k$-Motzkin paths with $k$ even}\label{even}
In the case when $k$ is even, the presence of the step $(1,-2)$ prevents us from using the framework in Section \ref{sec:main} to evaluate $\sum_{M\in\mathcal{M}_n^{(k)}}\prod_{i=1}^{|\mathbf{u}(M)|}f(\mathbf{u}(M)_i)$. Nevertheless, we are able to at least compute $|\mathcal{M}_n^{(k)}|$ using the kernel method. The first few terms of these sequences when $k=2,4,6,8,10$ are collected in the table below. None of these sequences are in OEIS.

\begin{center}
\begin{tabular}{|c|c|}
\hline
  $k$   &  $|\mathcal{M}_n^{(k)}|$ for $0\leq n\leq 7$ \\\hline
  2   &  1, 2, 17, 204, 2848, 43335, 697194, 11663971\\\hline
  4   &  1, 3, 66, 2100, 78399, 3202513, 138606469, 6245691198\\\hline
  6   &  1, 4, 164, 9837, 694906, 53797628, 4416325803, 377628587186\\\hline
  8   &  1, 5, 327, 31515, 3584682, 447231641, 59192155893, 8162250566928\\\hline
  10   &  1, 6, 571, 80482, 13406549, 2450879425, 475440187468, 96106360517372\\\hline
\end{tabular}
\end{center}

Note that in the formula for $\mu_n$ below, as $\frac {(k+1)n+1}2$ may not be an integer, we are using the generalised binomial coefficient $\binom{\alpha}{k}:=\frac1{k!}\prod_{i=0}^{k-1}(\alpha-i)$ for all $\alpha\in\mathbb{C}$ and $k\in\mathbb{Z}_{\geq0}$. 

\begin{theorem}
For every $n\in\mathbb{Z}_{\geq0}$ and every even $k\in\mathbb{Z}_+$, we have
\[|\mathcal{M}_{n}^{(k)}|=\sum_{i=0}^{2n}(-1)^{i}\mu_i\mu_{2n-i},\]
where \[\mu_n=\frac1{(k+1)n+1}\sum_{i=0}^{\floor{\frac n2}}\binom{\frac {(k+1)n+1}2}{n-i}\binom{n-i}{i}.\]
\end{theorem}
\begin{proof}
Fix an even $k\in\mathbb{Z}_+$, we proceed using the kernel method to compute $|\mathcal{M}_n^{(k)}|$.

Let $\mathfrak{M}_n^{(k)}$ be the set of $n$-step lattice paths with step set $\{(1,2k),(1,k-1),(1,-2)\}$ that start from $(0,0)$ and never go below the $x$-axis. For each $M\in\mathfrak{M}_n^{(k)}$, let $h(M)$ be the $y$-coordinate of the last vertex in $M$. Let $M_n(y)=\sum_{M\in\mathfrak{M}_n^{(k)}}y^{h(M)}$ for $n\geq 0$, and $M(x,y)=\sum_{n\geq0}x^nM_n(y)$. From the definitions above, we have the recurrence \[M_{n+1}(y)=M_n(y)\left(\frac1{y^2}+y^{k-1}+y^{2k}\right)-\frac1{y^2}\left([y^0]M_n(y)\right)-\frac1y\left([y^1]M_n(y)\right).\]
Substituting this recurrence into $M(x,y)$, we have
\begin{align*}
M(x,y)&=\sum_{n\geq0}x^nM_n(y)\\
&=1+\sum_{n\geq0}\left(M_n(y)\left(\frac1{y^2}+y^{k-1}+y^{2k}\right)-\frac1{y^2}\left([y^0]M_n(y)\right)-\frac1y\left([y^1]M_n(y)\right)\right)x^{n+1}\\
&=1+x\left(\frac1{y^2}+y^{k-1}+y^{2k}\right)M(x,y)-\sum_{n\geq0}\left(\frac1{y^2}\left([y^0]M_n(y)\right)+\frac1y\left([y^1]M_n(y)\right)\right)x^{n+1}.
\end{align*}
Let $F_0(x)=\sum_{n\geq0}\left([y^0]M_n(y)\right)x^{n+1}$ and $F_1(x)=\sum_{n\geq0}\left([y^1]M_n(y)\right)x^{n+1}$, then 
\begin{equation}\label{eq:dagger}\tag{$\dagger$}
\left(1-x\left(\frac1{y^2}+y^{k-1}+y^{2k}\right)\right)M(x,y)=1-\frac1{y^2}F_0(x)-\frac1yF_1(x).
\end{equation}
Note that $[y^0]M_n(y)$ is exactly the number of lattice paths from $(0,0)$ to $(n,0)$ with step set $\{(1,2k),(1,k-1),(1,-2)\}$ that never go below the $x$-axis, so \[|\mathcal{M}_n^{(k)}|=[y^0]M_{(k+1)n}(y)=[x^{(k+1)n+1}]F_0(x).\]

Consider the kernel equation $1-x\left(\frac1{y^2}+y^{k-1}+y^{2k}\right)=0$, which can be rewritten as $y^2=x(1+y^{k+1}+y^{2k+2})$. Let $y_1(x),y_2(x)$ be the two small roots of this kernel equation. By the conjugate principle (see \cite{BF,FS}), $y_1(x)=\sum_{n\geq1}a_nx^{\frac n2}$ and $y_2(x)=\sum_{n\geq 1}(-1)^na_nx^{\frac n2}$, where the coefficients $a_n$ come from the unique power series solution $y(z)=\sum_{n\geq 1}a_nz^n$ given by Lemma \ref{lift} to the functional equation $y=z(1+y^{k+1}+y^{2k+2})^{\frac12}$. Using Lemma \ref{lift}, we get
\[a_n=[z^n]y(z)=\frac1n[z^{n-1}](1+z^{k+1}+z^{2k+2})^\frac n2.\]
Observe that if $a_n\not=0$, then $k+1\mid n-1$. For simplicity, let $\mu_n=a_{(k+1)n+1}$. Then we have
\begin{align*}
\mu_n&=[z^{(k+1)n+1}]y(z)\\
&=\frac1{(k+1)n+1}[z^{(k+1)n}](1+z^{k+1}+z^{2k+2})^\frac {(k+1)n+1}2=\frac1{(k+1)n+1}[z^n](1+z+z^2)^\frac {(k+1)n+1}2\\
&=\frac1{(k+1)n+1}\sum_{i=0}^{\floor{\frac n2}}\binom{\frac {(k+1)n+1}2}{n-i}[z^n](z+z^2)^{n-i}=\frac1{(k+1)n+1}\sum_{i=0}^{\floor{\frac n2}}\binom{\frac {(k+1)n+1}2}{n-i}[z^i](1+z)^{n-i}\\
&=\frac1{(k+1)n+1}\sum_{i=0}^{\floor{\frac n2}}\binom{\frac {(k+1)n+1}2}{n-i}\binom{n-i}{i}.
\end{align*}

From the theory of the kernel method \cite{BF,FS}, we may substitute the two small roots $y=y_1(x),y_2(x)$ into (\ref{eq:dagger}), and obtain the following system of equations as both sides of (\ref{eq:dagger}) equal 0.
\begin{align*}
   F_0(x)+y_1(x)F_1(x) &= y_1^2(x)\\
   F_0(x)+y_2(x)F_1(x) &= y_2^2(x)
\end{align*}
Solving this system, we get $F_0(x)=-y_1(x)y_2(x)$, so 
\begin{align*}
|\mathcal{M}_n^{(k)}|&=[x^{(k+1)n+1}]F_0(x)=-[x^{(k+1)n+1}]y_1(x)y_2(x)\\
&=-[x^{(k+1)n+1}]\left(\sum_{i\geq0}\mu_ix^{\frac{(k+1)i+1}2}\right)\left(\sum_{j\geq0}(-1)^{(k+1)j+1}\mu_jx^{\frac{(k+1)j+1}2}\right)\\
&=-\sum_{i=0}^{2n}(-1)^{(k+1)i+1}\mu_i\mu_{2n-i}=\sum_{i=0}^{2n}(-1)^{i}\mu_i\mu_{2n-i},
\end{align*}
as required.
\end{proof}


\section{Concluding remarks}
Following the discussions in Section \ref{sec:motzkin}, a natural open question is the weighted enumeration of $k$-Motzkin paths when $k$ is even. 
\begin{question}
For $n\in\mathbb{Z}_{\geq0}$, even $k\in\mathbb{Z}_+$, and certain weight function $f:\mathbb{Z}_+\to\mathbb{C}$, compute
\[\sum_{M\in\mathcal{M}_n^{(k)}}\prod_{i=1}^{|\mathbf{u}(M)|}f(\mathbf{u}(M)_i),\]
and/or find a function equation that its generating function satisfies. 
\end{question}

In this paper we focused on the weighted enumerations of three particular generalisations of Catalan paths, Schr\"{o}der paths and Motzkin paths, respectively, with respect to some linear weight functions. In many instances, the resulting sequences either appear as solutions to other enumeration problems, or provide combinatorial interpretations for some OEIS entries defined purely algebraically. In the several aforementioned cases when sequences coming from different combinatorial settings coincide, it would be nice to find combinatorial explanations despite the relatively simple generating function proofs. Furthermore, the results in Section \ref{sec:main} apply much more generally to many other possible generalisations, other families of lattices paths and other weight functions, likely leading to further connections with other enumeration problems and OEIS sequences.

Finally, we pose a related weighted enumeration problem where a slightly different weight is assigned to each lattice path. Several known applications will be described after.
\begin{question}\label{q:2}
Compute for $k,m\in\mathbb{Z}_+$
\[\sum_{C\in\mathcal{C}^{(k)}_n}\prod_{i=2}^{|\mathbf{u}(C)|}\left(1+m\sum_{j=i}^{|\mathbf{u}(C)|}\mathbf{u}(C)_j\right),\]
and/or find a nice function equation that its generating function satisfies.
\end{question}
If $k=1$, it was shown in \cite{Y} that this sum equals the number of metasylvester classes of $m$-multiparking functions as defined by Novelli and Thibon in \cite{NT}. A recurrence formula for this sequence and an intriguing equation that it satisfies were also proved in \cite{Y}, but no explicit formula or nice functional equation for its generating function was found. If, in addition, that $m=1$, then it was shown in \cite{Y} that this equals the number of parking functions of size $n$ whose parking outcomes avoid the pattern 312, and that it coincides with OEIS sequence A132624, which comes from an entirely different context.

If $m=1$, it was shown in \cite{Y} that this equals the number of metasylvester classes of $k$-parking functions as defined in \cite{NT}, though in this case even a nice recurrence formula is not known.

\bibliographystyle{abbrv}
\bibliography{bibliography}

\begin{thebibliography}{10}

\bibitem{AP}
A.~Adeniran and L.~Pudwell.
\newblock Pattern avoidance in parking functions.
\newblock {\em Enumerative Combinatorics and Applications}, 3(3):S2R17, 2023.

\bibitem{BF}
C.~Banderier and P.~Flajolet.
\newblock Basic analytic combinatorics of directed lattice paths.
\newblock {\em Theoretical Computer Science}, 281(1--2):37--80, 2002.

\bibitem{C}
L.~Carlitz.
\newblock Enumeration of two-line arrays.
\newblock {\em The Fibonacci Quarterly}, 11(2):113--130, 1973.

\bibitem{FS}
P.~Flajolet and R.~Sedgewick.
\newblock {\em Analytic Combinatorics}.
\newblock Cambridge University Press, 2009.

\bibitem{GPPT}
N.~Gabriel, K.~Peske, L.~Pudwell, and S.~Tay.
\newblock Pattern avoidance in ternary trees.
\newblock {\em Journal of Integer Sequences}, 15(1):5, 2012.

\bibitem{GJ}
I.~Goulden and D.~Jackson.
\newblock {\em Combinatorial Enumeration}.
\newblock Dover Publications, 2004.

\bibitem{MS2}
T.~Mansour and Y.~Sun.
\newblock Bell polynomials and $k$-generalized {D}yck paths.
\newblock {\em Discrete Applied Mathematics}, 156(12):2279--2292, 2008.

\bibitem{MS1}
T.~Mansour and Y.~Sun.
\newblock Dyck paths and partial {B}ell polynomials.
\newblock {\em Australasian Journal of Combinatorics}, 42:285--297, 2008.

\bibitem{NT}
J.-C. Novelli and J.-Y. Thibon.
\newblock Hopf algebras of $m$-permutations, $(m+1)$-ary trees, and $m$-parking functions.
\newblock {\em Advances in Applied Mathematics}, 117:102019, 2020.

\bibitem{S}
S.~J. Selkirk.
\newblock On a generalisation of $k$-{D}yck paths.
\newblock Master's thesis, Stellenbosch University, 2019.

\bibitem{Sun}
Y.~Sun.
\newblock Potential polynomials and {M}otzkin paths.
\newblock {\em Discrete Mathematics}, 309(9):2640--2648, 2009.

\bibitem{Y}
J.~Yan.
\newblock Results on pattern avoidance in parking functions.
\newblock {\em Enumerative Combinatorics and Applications}, 5(1):S2R2, 2024.

\bibitem{YZY}
L.~Yang, Y.-Y. Zhang, and S.-L. Yang.
\newblock Enumeration of the {M}otzkin paths above a line of rational slope.
\newblock {\em Discrete Mathematics}, 347(7):114013, 2024.

\bibitem{YJ}
S.-L. Yang and M.-Y. Jiang.
\newblock The $m$-{S}chr{\"{o}}der paths and $m$-{S}chr{\"{o}}der numbers.
\newblock {\em Discrete Mathematics}, 344(2):112209, 2021.

\bibitem{YJ2}
S.-L. Yang and M.-Y. Jiang.
\newblock Pattern avoiding problems on the hybrid $d$-trees.
\newblock {\em Journal of Lanzhou University of Technology}, 49(2):144--150, 2023.

\end{thebibliography}
\end{document}